\documentclass[11pt,a4paper,twoside]{article}
\usepackage{amssymb,amsthm}
\usepackage{amsmath}
\usepackage{setspace}
\usepackage[a4paper,top=33mm, bottom=27mm, left=20mm, right=20mm]{geometry}
\usepackage[small, margin=20pt]{caption}
\usepackage{color}
\usepackage[textsize=footnotesize,color=green!40]{todonotes}

\usepackage{eurosym}
\usepackage{palatino}
\usepackage{tikz}
\usepackage{verbatim}
\usepackage{fancybox}

\usepackage{algorithmic}
\usepackage{algorithm}
\floatname{algorithm}{Procedure}

\usepackage{soul}

\usepackage{hyperref}
\allowdisplaybreaks[4]

\allowdisplaybreaks[4]

\title{Fast Recoloring of Sparse Graphs}

\author{Nicolas Bousquet\thanks{Department of Mathematics and Statistics. McGill University. 845 Rue Sherbrooke Ouest, Montreal, Quebec H3A 0G4, Canada. Email:~\emph{nicolas.bousquet2@mail.mcgill.ca}.} \and Guillem Perarnau\thanks{School of Computer Science. McGill University. 845 Rue Sherbrooke Ouest, Montreal, Quebec H3A 0G4, Canada. Email:~\emph{guillem.perarnaullobet@mail.mcgill.ca}.}}
\date{\today}

\theoremstyle{plain}
\newtheorem{theorem}{Theorem}
\newtheorem{lemma}[theorem]{Lemma}
\newtheorem{observation}{Observation}

\newtheorem{claim}{Claim}

\newtheorem{proposition}[theorem]{Proposition}
\newtheorem{corollary}[theorem]{Corollary}

\theoremstyle{definition}

\renewenvironment{proof}[1][Proof]{\begin{trivlist}
\item[\hskip\labelsep {\textit{#1}.}]}{\qed\end{trivlist}}

\newcommand{\C}{\mathcal{C}}

\newcommand{\eps}{\varepsilon}
\newcommand{\Poly}{\text{Poly}}
\newcommand{\diam}{\text{diam}}
\newcommand{\mad}{\text{mad}}

\definecolor{red}{RGB}{255,0,0}
\definecolor{blue}{RGB}{0,0,255}

\begin{document}

\pagenumbering{arabic}

\setcounter{section}{0}

\maketitle

\begin{abstract}
In this paper, we show that for every graph of maximum average degree bounded away from $d$, any $(d+1)$-coloring can be transformed into any other one within a polynomial number of vertex recolorings so that, at each step, the current coloring is proper. In particular, it implies that we can transform any $8$-coloring of a planar graph into any other $8$-coloring with a polynomial number of recolorings. These results give some evidence on a conjecture of Cereceda et al~\cite{Cereceda09} which asserts that any $(d+2)$ coloring of a $d$-degenerate graph can be transformed into any other  one using a polynomial number of recolorings.

We also show that any $(2d+2)$-coloring of a $d$-degenerate graph can be transformed into any other  one with a linear number of recolorings.
\end{abstract}

\section{Introduction}

Reconfiguration problems consist in finding step-by-step transformations between two feasible solutions such that all intermediate states are also feasible. Such problems model dynamic situations where a given solution is in place and has to be modified, but no property disruption can be afforded. Recently, reconfigurations problems have raised a lot of interest in the context of constraint satisfaction problems~\cite{BonsmaMNR14,Gopalan09} and of graph invariants like independent sets~\cite{ItoKO14}, dominating sets~\cite{BonamyB14a,SuzukiMN14} or vertex colorings~\cite{BonamyJ12,BonsmaC07}.

In this paper $G=(V,E)$ is a graph where $n$ denotes the order of $V$ and $k$ is an integer. For standard definitions and notations on graphs, we refer the reader to~\cite{Diestel}.
A \emph{proper $k$-coloring} of $G$ is a function $f : V(G) \rightarrow \{ 1,\ldots,k \}$ such that, for every $xy\in E$, $f(x)\neq f(y)$. Throughout the paper we will only consider proper colorings. In the following, we will omit the proper for brevity. The \emph{chromatic number} $\chi(G)$ of a graph $G$ is the smallest $k$ such that $G$ admits a $k$-coloring.
Two $k$-colorings are \emph{adjacent} if they differ on exactly one vertex. The \emph{$k$-recoloring graph of $G$}, denoted by $\C_k(G)$ and defined for any $k\geq \chi(G)$, is the graph whose vertices are $k$-colorings of $G$, with the adjacency condition defined above. Note that two colorings equivalent up to color permutation are distinct vertices in the recoloring graph.
The graph $G$ is \emph{$k$-mixing} if $\C_k(G)$ is connected. Cereceda, van den Heuvel and Johnson characterized the $3$-mixing graphs and provided an algorithm to recognize them~\cite{Cereceda09,CerecedaHJ11}. The easiest way to prove that a graph $G$ is not $k$-mixing is to exhibit a \emph{frozen} $k$-coloring of $G$, \emph{i.e.} a coloring in where every vertex is adjacent to vertices of all other colors. Such a coloring is an isolated vertex in $\C_k(G)$.

Deciding whether a graph is $k$-mixing is $\mathbf{PSPACE}$-complete for $k \geq 4$~\cite{BonsmaC07}. The \emph{$k$-recoloring diameter} of a $k$-mixing graph is the diameter of $\C_k(G)$. In other words, it is the minimum $D$ for which any $k$-coloring can be transformed into any other one through a sequence of at most $D$ adjacent $k$-colorings.
Bonsma and Cereceda~\cite{BonsmaC07} proved that there exists a family of graphs and an integer $k$ such that, for every graph $G$ in the family there exist two $k$-colorings whose distance in the $k$-recoloring graph is finite and super-polynomial in $n$.
Though, the diameter of the $k$-recoloring may be polynomial when we restrict to a well-structured class of graphs and $k$ is large enough. Graphs with bounded degeneracy are natural candidates.

The diameter of the $k$-recoloring graphs has been already studied in terms of the degeneracy of a graph.
It was shown independently by Dyer et al~\cite{dyer2006randomly} and by Cereceda et al.~\cite{Cereceda09} that for any $(d-1)$-degenerate graph $G$ and every $k \geq d+1$, $\C_{k}(G)$ is connected ($\diam (\C_{k}(G))<\infty$). Moreover, Cereceda~\cite{Cereceda} also showed that for any $(d-1)$-degenerate graph $G$ and every $k \geq 2d-1$, we have $\diam (\C_{k}(G))=O(n^2)$.
%
Cereceda et al. conjectured in 2009~\cite{Cereceda09} that,  for any $(d-1)$-degenerate graph $G$ and every $k \geq d+1$, we have $\diam (\C_{k}(G))=O(n^2)$. No general result is known so far on this conjecture, but several particular  cases have been treated in the last few years. Bonamy et al.~\cite{BonamyJ12} showed that for every $(d-1)$-degenerate chordal graph and every $k \geq d+1$, $\diam (\C_{k}(G))=O(n^2)$, improving the results of~\cite{Cereceda09,dyer2006randomly}. This result was then extended to graphs of bounded treewidth by Bonamy and Bousquet in~\cite{BonamyB13}. Unfortunately, all these results are based on the existence of an underlying tree structure. This leads to nice proofs but new ideas are required to extend these results to other classes of graphs.

\paragraph{Our results.}
In Section~\ref{sec:linear}, we show that the Cereceda's quadratic bound on the recoloring diameter can be improved into a linear bound if one more color is available. More precisely we show that for every $(d-1)$-degenerate graph $G$ and every $k \geq 2d$, the recoloring diameter of $G$ is at most $d n$.

In Section~\ref{sec:mad}, we study the $k$-recoloring diameter from another invariant of graphs related to degeneracy: the maximum average degree. The \emph{maximum average degree} of $G$, denoted by $\mad(G)$, is the maximum average degree of a (non-empty) induced subgraph $H$ of $G$. We prove that for every integer $d\geq 1$ and for every $\eps>0$, there exists $c=c(d,\eps)\geq 1$ such that for every graph $G$ satisfying $\mad(G) < d-\eps$ and for every $k \geq d+1$, $\diam (\C_{k}(G))=O(n^c)$. The proof goes as follows. We first show that the vertex set can be partitioned into a logarithmic number of sparse sets. Using this partition, we show that one color can be eliminated after a polynomial number of recolorings and then we finally conclude by an iterative argument.

Since every planar graph $G$ satisfies $\mad(G)\leq 6$, our result implies that for every $k \geq 8$ the diameter of the $k$-recoloring graph of $G$ is polynomial in $n$. Bousquet and Bonamy observed in~\cite{BonamyB14} that $k\geq 7$ is needed to obtain such a conclusion and conjectured that $k=7$ is enough (this is the planar graph version of the conjecture raised by Cereceda et al.~\cite{Cereceda09} for degenerated graphs). We also discuss the limitations of our approach by showing that it cannot provide a polynomial bound on the diameter of the $7$-recoloring graph of a planar graph. Finally, we also mention other consequences of our result to triangle-free planar graphs.

The degeneracy is closely related to the maximum average degree: a graph $G$ satisfying $\mad(G) \leq d$ is $d$-degenerate and every $d$-degenerate graph has maximum average degree at most $2d$ (see e.g. Proposition 3.1 of~\cite{NesetrilPOM}). Using the latter  inequality, one can deduce from our result that if $G$ has degeneracy $d-1$, the diameter of the $2d$-recoloring graph of $G$ is polynomial in $n$. However, as the first part of our paper shows, better results can be attained in such case.

\section{Linear diameter with \texorpdfstring{$2d$}{2d} colors}\label{sec:linear}

Let us first set some basic notations. Let $X$ be a subset of $V$. The \emph{size} $|X|$ of $X$ is its number of elements.
Let $G=(V,E)$ be a graph. For any coloring $\alpha$ of $G$, we denote by $\alpha(H)$ the set of colors used by $\alpha$ on the subgraph $H$ of $G$. The \emph{neighborhood} of a vertex $x$, denoted by $N(x)$, is the subset of vertices $y$ such that $xy \in E$. The \emph{length} of a path $P$ is its number of edges and its size, denoted by $|P|$, is its number of vertices. The \emph{distance} between two vertices $x$ and $y$, denoted $d(x,y)$, is the minimum length of a path between these two vertices. When there is no path, the distance is considered to be infinite. The \emph{distance} between two $k$-colorings of $G$ is implicitly the distance between them in the $k$-recoloring graph $\C_k(G)$.  The \emph{diameter} of $G$ is the maximum, over all the pairs $u,v\in V(G)$, of the distance between $u$ and $v$.


\begin{theorem}\label{thm:lineardiam}
 For every $(d-1)$-degenerate graph $G$ on $n$ vertices and every $k \geq 2d$, $\diam (\C_{k}(G)) \leq d n$. Even stronger, there exists a recoloring procedure where every vertex is recolored at most $d$ times.
\end{theorem}
\begin{proof}
Let $\alpha$ and $\beta$ be two $k$-colorings. We will show by induction on the number of vertices that there exists a recoloring procedure that transforms $\alpha$ into $\beta$ and where every vertex is recolored at most $d$ times. If $n=1$ the result is obviously true. Let $G$ be a $(d-1)$-degenerate graph on $(n+1)$ vertices and let $u$ be a vertex of degree at most $d-1$. Consider $G'$ to be the graph induced by $V \setminus u$. Let us denote by $\alpha'$ and $\beta'$ the restrictions of $\alpha$ and $\beta$ to $G'$. By induction, the coloring $\alpha'$ can be transformed into $\beta'$ so that every vertex is recolored at most $d$ times and at every step, the $k$-coloring is proper in $G'$.

Since $u$ has at most $d-1$ neighbors and since each vertex in $G'$ is recolored at most $d$ times, the neighbors of $u$ are recolored $\ell\leq d(d-1)$ times in this sequence. Let $t_1, \dots,t_{\ell}$ be the times in the recoloring sequence when a neighbor of $u$ changes its color. For any time $t$ in the sequence, let $c_t$ be the new color assigned at this time.


Consider again the initial graph $G$. Let us now try to add some recolorings of the vertex $u$ in the sequence of recolorings obtained for $G'$ to guarantee that the $k$-colorings are proper in $G$. We claim that the vertex $u$ can be inserted in thee recoloring sequence of $G'$ with the addition of at most $d$ new recoloring steps that change the color of $u$. Consider the following recoloring algorithm: at each step of the recoloring process, some vertex $v$ is recolored from color $a$ to color $b$. If $v$ is not a neighbor of $u$ or if the current color of $u$ is not $b$, the obtained coloring is still proper in $G$ and we do not perform any recoloring of $u$.
Assume now that $v\in N(u)$ and that the color of $u$ is $b$. This happens at some time $t_i$, with $i\leq \ell$. In this case, we add a new recoloring step in our sequence right before the recoloring of $v$ at time $t_i$,  in which we change the color of $u$. In order to maintain the proper coloring, we want to assign to $u$ a color distinct from the colors in $N(u)$ (there are at most $d-1$ different colors there). So there remain at least $k-(d-1) \geq d+1$ choices of colors for $u$ that do not create monochromatic edges. Thus, we assign to $u$ a color distinct from $c_{t_i},\ldots,c_{t_{i+d-1}}$. By choosing this color, we make sure that $u$ will require no recoloring before time $t_{i+d}$ in the sequence.

Let $s$ be the number of recolorings of $u$ and let $t_{i_1},\dots, t_{i_s}$ the corresponding recoloring times in the original sequence. By the construction of the new sequence, observe that $i_{j+1}-i_j\geq d$ for every $j< s$. Since $\ell\leq d(d-1)$ and $i_s\leq \ell$, we have that $s\leq d-1$. Observe that at the end of the procedure we may have to change the color of $u$ to $\beta(u)$ if it is not its current color. Hence, the recoloring of $V \setminus u$ can be extended to $V$  and we recolor each vertex at most $d$ times, which concludes the proof.
\end{proof}

\section{Recoloring sparse graphs}\label{sec:mad}

The \emph{maximum average degree} of a graph $G$ is defined as
$$
\mad(G)=\max_{\emptyset \neq H\subseteq G} \frac{2|E(H)|}{|V(H)|}\;.
$$

We will prove the following theorem that relates the maximum average of the graph with the diameter of its recoloring graphs.
\begin{theorem}\label{thm:mad_recol}
For every integer $d\geq 0$ and for every $\eps>0$, there exists $c=c(d,\eps)\geq 1$ such that for every graph $G$ on $n$ vertices satisfying $\mad(G)\leq d-\eps$ and for every $k\geq d+1$, we have $\diam(\C_{k}(G))=O(n^c)$.
\end{theorem}


For every graph $G$ and every $t$-partition $\{V_1, \dots, V_t\}$ of the vertex set of $G$, we consider the following induced subgraphs for every $i\leq t$,
$$
G_i=G\left[\cup_{j\geq i} V_j\right]\;.
$$
A \emph{$t$-partition of degree $\ell$ of $G$} is a partition $\{V_1, \dots, V_t\}$ of the vertex set of $G$ such that every vertex $v\in V_i$ has degree at most $\ell$ in $V(G_i)$. The \emph{level function} of a partition, denoted by $L: V(G) \longrightarrow \{ 1,\ldots,t \}$, labels each vertex with its corresponding part of the partition, that is $L(u)=i$ for every $u \in V_i$.

The existence of a $t$-partition of degree $(d-1)$ is crucial in the proof of our theorem. Let us briefly explain why. Fix $k\geq d+1$. For any $k$-coloring $\alpha$ of $G$ and for every vertex $v\in V$, there exists at least one color $a \neq \alpha(v)$ that does not appear in $N_{G_{L(v)}}(v)$. Indeed the vertex $v$ has at most $(d-1)$ neighbors in $G_{L(v)}$ and there are $k\geq d+1$ colors. Thus, we can always change the color $\alpha(v)$ by $a$ without creating any monochromatic edge in $G_i$.
Nevertheless, notice that this recoloring may create monochromatic edges in $G$. The following lemma will take care of them by showing that a polynomial number of recolorings is enough to ensure that the recoloring of $v$ with color $a$ does not create any monochromatic edge in $G$.

We say that two colorings $\alpha$ and $\beta$ \emph{agree} on some subset $X$ if $\alpha(x)= \beta(x)$ for every $x \in X$.

\begin{lemma}\label{lem:notmanyrecol}
Suppose that $G$ admits a $t$-partition of degree $\ell$.
For every $v\in V$ and every $(\ell+2)$-coloring $\alpha$, we can change the color of $v$ by recoloring each vertex at most $\ell^{L(v)}$ times. Moreover, the current $(\ell+2)$-coloring agrees with $\alpha$ in $V(G_{L(v)})\setminus \{v\}$ at any recoloring step.
\end{lemma}
\begin{proof}
Consider a total order $\prec$ on the set of vertices such that if $u \prec w$ then $L(u) \leq L(w)$.
The proof is based on a recursive recoloring algorithm. Let us first give a few definitions. A procedure $C$ \emph{calls} a procedure $D$ if $D$ is started during the procedure $C$. In this case we also say that $D$ is a \emph{recursive call} of $C$. A procedure $D$ is \emph{generated} by $C$ if there exists a sequence of procedures $C=C_1, C_2,\dots, C_t=D$ such that $C_i$ calls $C_{i+1}$ for every $i<t$.

We consider Algorithm~\ref{alg:recoloring} which has as an input a tuple $(\gamma,P)$, where $\gamma$ is a coloring of $G$ and $P$ is a list of vertices that forms a path in $G$. We will call Algorithm~\ref{alg:recoloring} with input $(\alpha,\{v\})$.
\begin{algorithm}
\caption{Recoloring Algorithm}
\label{alg:recoloring}
\begin{algorithmic}[PERF]
\STATE {\bf Input:} A coloring $\gamma$ of $G$, a list $P$ of vertices.
\STATE {\bf Output:} A coloring $\gamma'$ of $G$ which agrees with $\gamma$ on $V(G_{L(u)}) \setminus u$ where $u$ is the last element of $P$. Moreover $\gamma(u) \neq \gamma'(u)$. \vspace{10pt}
\STATE Let $u$ be the last element of $P$. \hfill {\scriptsize $u$ is the current vertex of the procedure.}
\STATE Let $a$ be a color not in $\gamma\left(\{u\}\cup N_{G_{L(u)}}(u)\right)$.\hfill {\scriptsize Such a color exists and is the \emph{target color} for $u$.}
\STATE Let $\gamma' = \gamma$. \hfill{\scriptsize $\gamma'$ is the current coloring.}
\STATE Let $X=\{ v_1\succ\ldots\succ v_s \}$ be the set of neighbors of $u$ in $\cup_{j < L(u)} V_j$.
\FOR{$v_i \in X$ with $i$ increasing}
   \IF{$v_i$ is colored with $a$}
  \STATE Add $v_i$ at the end of $P$.
  \STATE $\gamma' \leftarrow$ Algorithm~\ref{alg:recoloring} with input $(\gamma',P)$.  \hfill {\scriptsize The color of $v_i$ is now different from $a$.}
  \STATE Delete $v_i$ from the end of $P$.
  \ENDIF
\ENDFOR
\STATE Change the color of $u$ to $a$ in $\gamma'$.
\STATE {\bf Output} $\gamma'$
\end{algorithmic}
\end{algorithm}

The vertex $u$ in the last position of the list $P$ in the procedure, will be called the \emph{current} vertex of the procedure.

Let us first state a few immediate remarks concerning this algorithm. In each recursive call, we add one vertex in the list $P$. By construction, the vertex added in $P$ in the recursive call is a neighbor of the current vertex $u$ and has level strictly smaller than $u$. Since in any procedure $C$ the unique recolored vertex is the current vertex $u$, an immediate induction argument ensures that any recolored vertex in procedures generated by $C$ has level strictly smaller than $L(u)$. So we have the following:

\begin{observation}\label{obs:algo1}
If the procedure $C$ with input $(\gamma,P)$ makes some recursive calls, then the size of $P$ increases in these calls. Moreover, the level of the vertex $v_i$ added at the end of $P$ during $C$, is strictly smaller than the level of the current vertex $u$ and both vertices are adjacent. \\
This implies that for every vertex $w$ recolored in a procedure generated by $C$ we have $L(w)<L(u)$, \emph{i.e.} the coloring output by a procedure with current vertex $u$ agrees with $\gamma$ on $V(G_{L(u)}) \setminus u$.
\end{observation}

Let us now prove that Algorithm~\ref{alg:recoloring} ends, that it makes the right amount of recolorings and that it is correct.

\paragraph{Termination and number of recolorings in Algorithm~\ref{alg:recoloring}.}
Each call of Algorithm~\ref{alg:recoloring} creates at most $n$ recursive calls (we have a priori no good upper bound on the number of neighbors of $u$ in $\cup_{j<L(u)} V_j$). Since the level of the current vertex $u$ decreases at every recursive call, the depth of the recursion is at most $L(u)\leq t$. This implies that Algorithm~\ref{alg:recoloring} will terminate in at most $n^{t}$ iterations.
We need an additional argument to show that the number of recolorings is at most $\ell^t$ as stated in Lemma~\ref{lem:notmanyrecol}.

Notice that the number of recolorings is exactly the number of procedures since every procedure $C$ only recolors one vertex once, the current one in $C$.
Recursive calls made in a procedure where $u$ is the current vertex are called recursive calls of $u$. If we can bound the number of procedures where $v$ is the current vertex, then we can bound the number of recolorings of $v$. We say that a procedure $C$ is \emph{generated by $u$} if a procedure with current vertex $u$ generates $C$.
We will show that the sequence of paths used as an input of successive calls of Algorithm~\ref{alg:recoloring}, is lexicographically strictly decreasing (in particular two procedures cannot have the same path $P$). Then we will prove that the number of paths passing through any vertex is bounded. These two facts suffice to provide a meaningful upper bound on the number of recolorings.

Recall that the vertices of $G$ are equipped with a total order $\prec$. A path $P_1$ is lexicographically smaller than $P_2$, denoted by $ P_1 \prec_l P_2$ if:
\begin{itemize}
 \item $P_2$ is empty and $P_1$ is not. 
 \item The first vertex of $P_1$ is smaller than the first vertex of $P_2$.
 \item The first vertices of both paths are the same  and the path $P_1$ without its first vertex is lexicographically smaller than the path $P_2$ without its first vertex.
\end{itemize}
Informally, we compare the first vertex of each path (which in our case will be the largest) and if they are not equal, the largest path is the one with the largest vertex; otherwise we compare the remaining paths. Notice that if $P_2$ is contained in the first positions of a path $P_1$, then $P_1 \prec_l P_2$. In particular, with this definition, the empty path is the largest one.

The \emph{path of the procedure $C$}, denoted by $P_C$, is the path $P$ given as an input of the procedure $C$.
\begin{claim}\label{cla:orderpaths}
If procedure $D$ is initiated after procedure $C$, then $P_D \prec_l P_C$.
\end{claim}
\begin{proof}
First note that if $D$ is called by $C$ then $P_D \prec_l P_C$. Indeed, the path $P_C$ is contained in the first positions of $P_D$. Consider now two procedures $C$ and $D$ such that $D$ is generated by $C$. An immediate induction argument using the previous observation ensures that $P_D \prec_l P_C$.

So we may assume that $D$ is not generated by $C$. Let us denote by $I$ the initial procedure. Recall that all the procedures are generated by $I$ and that the procedures are organized in a tree structure. So there exist a unique sequence $\mathcal{S}_1:I=C_1, C_2,\dots, C_{t_1}=C$ such that $C_{j}$ calls $C_{j+1}$ for every $j < t_1$ and a unique sequence $\mathcal{S}_2:I=D_1, D_2,\dots, D_{t_2}=D$ such that $D_j$ calls $D_{j+1}$ for every $j < t_2$. Let us denote by $B$ the last common procedure in $\mathcal{S}_1$ and $\mathcal{S}_2$. Since $D$ is not generated by $C$, $\mathcal{S}_1$ is not included in $\mathcal{S}_2$ and then $B$ is not the last element of $\mathcal{S}_1$ or $\mathcal{S}_2$. Let us denote by $B_C$ the procedure called by $B$ in $\mathcal{S}_1$ and by $B_D$ the procedure called by $B$ in $\mathcal{S}_2$. We have:
\begin{itemize}
\item $B_C$ and $B_D$ are called by $B$ in this order (otherwise $D$ would have been initiated before $C$),
\item either $B_C=C$ or $B_C$ generates $C$, and
\item either $B_D=D$ or $B_D$ generates $D$.
\end{itemize}
The previous observations ensure that $P_{D} \preceq_l P_{B_D}$. Thus, it suffices to show that $P_{B_D} \prec_l P_{C}$. Since $B_C$ and $B_D$ are procedures called by $B$, the corresponding paths $P_{B_C}$ and $P_{B_D}$ are both $P_B$ plus a last additional vertex, denoted respectively by $v_{B_C}$ and $v_{B_D}$.  Since $B_C$ is called before $B_D$, by construction of Algorithm~\ref{alg:recoloring} we have $v_{B_D} \prec v_{B_C}$. Notice that $C$ is generated by $B_C$, which implies that $P_{B_C}$ is contained in the first $|P_{B_C}|$ positions of $P_C$. So the path $P_C$ is lexicographically larger than $P_{B_D}$: they coincide in the first $|P_B|$ positions and at the first position where they differ we have $v_{B_D}\prec v_{B_C}$.
\end{proof}

A path $P=(u_1,\dots, u_{s})$ is \emph{level-decreasing} if $L(u_i)>L(u_{i+1})$ for every $i< s$.
Observation~\ref{obs:algo1} ensures that $P_C$ is a level-decreasing path for any procedure $C$.

\begin{claim}\label{cla:numberpaths}
 The number of level-decreasing paths between two vertices $u$ and $w$ in different levels is at most $\ell^{i-1}$ where $i= |L(u) - L(w)|$.
\end{claim}
\begin{proof}
 Without loss of generality, we may assume $L(w)<L(u)$. Let us prove the claim by induction on $i$. If $i=1$, then there is at most one level-decreasing path between $u$ and $w$ which is the edge $uw$ if it exists. Assume now that $L(u)-L(w)=i$. By the definition of a $t$-partition of degree $\ell$, the vertex $w$ has at most $\ell$ neighbors in $G_{L(w)}$, and, in particular, $s \leq \ell$ neighbors in $\cup_{j = L(w)+1}^{L(u)-1} V_j$. Let us denote by $w_1,\ldots,w_{s}$ these neighbors of $w$. Notice that $1\leq L(u)-L(w_j) \leq i-1$ for every $w_j$. Since $P$ is a level-decreasing path from $u$ to $w$, the before last element of $P$ should be in $\{w_1,\ldots,w_s\}$. By induction, for every $w_j$, there are at most $\ell^{i-2}$ level-decreasing paths from $u$ to $w_j$. Therefore, the number of level-decreasing paths from $u$ to $w$ is at most
\[ \sum_{j=1}^{s} \ell^{i-2} \leq \ell^{i-1} \;,\]
which concludes the proof of Claim~\ref{cla:numberpaths}.
\end{proof}
Let $I$ be the initial procedure.  Recall that $P_I=\{ v \}$. Since each procedure $C$ is generated by $I$, the first vertex in the path $P_C$ is $v$. By Claim~\ref{cla:orderpaths}, the sequence of paths used as an input of successive calls of Algorithm~\ref{alg:recoloring} is lexicographically strictly decreasing. By Claim~\ref{cla:numberpaths}, the number of level-decreasing paths from the vertex $v$ to any given $w$ is at most $\ell^{i-1}$, where $i=L(v) - L(w)\leq L(v)$. Since the unique recolored vertex in each procedure is the current vertex $u$, we obtain that for every $v\in V$ and every $(\ell+2)$-coloring $\alpha$ we can change the color of $v$ by recoloring each vertex at most $\ell^{L(v)}$ times.

\paragraph{Correctness of Algorithm~\ref{alg:recoloring}.}
Let us now show that if the initial coloring is proper, then at any step the current coloring is also proper. We have already seen that in each procedure $C$ the unique recolored vertex is the current vertex $u$ (the last vertex in $P$) and in any recursive call of $u$, the current vertex $w$ satisfies $L(w)<L(u)$.


Now, let us see that when $u$ is recolored in procedure $C$ with color $a$, no neighbor of $u$ has color $a$. Color $a$ is chosen in Algorithm~\ref{alg:recoloring} such that no neighbor of $u$ in $V(G_{L(u)})$ is colored with $a$. Since, by Observation~\ref{obs:algo1}, the vertices of $V(G_{L(u)})$ are not recolored by any procedure generated by $C$, recoloring $u$ with $a$ does not create monochromatic edges in $V(G_{L(u)})$.

Let $v_1,\dots,v_s$ be the neighbors of $u$ in $\cup_{j < L(u)} V_j$ in decreasing order with respect to $\prec$.
Let $\gamma'_0=\gamma$ be the coloring used as an input of the procedure $C$ and, for every $i\leq s$, let $\gamma'_i$ be the coloring $\gamma'$ output by the procedure called by $C$ whose current vertex is $v_i$. Recall that when the recoloring of $u$ is performed, the current coloring is $\gamma'_s$. We will show that $\gamma'_s(v_i)\neq a$ for every $i\leq s$.

If $\gamma'_{i-1}(v_i)\neq a$, then we do not create any new procedure to change the color of $v_i$ and $\gamma'_i=\gamma'_{i-1}$. If $\gamma'_{i-1}(v_i)=a$, then $\gamma'_i$ is the output of Algorithm~\ref{alg:recoloring} with input parameters $\gamma=\gamma'_{i-1}$ and $P=(P_C,v_i)$. Since $v_i$ is now the last vertex of $P$, by construction of the algorithm, the coloring $\gamma'_{i}$ satisfies that $\gamma'_{i}(v_i)\neq a$.

It remains to show that the color of $v_i$ is not modified between $\gamma'_{i}$ and the final coloring $\gamma'_s$. For the sake of contradiction assume that $j_*\in \{i+1,\dots, s\}$ is the smallest integer $j$ such that $\gamma'_j(v_i)\neq \gamma'_i(v_i)$. This implies that $v_i$ is the current vertex of a procedure $D$ generated by the procedure corresponding to $v_{j_*}$. Hence, the vertex $v_{j_*}$ appears before than $v_{i}$ in $P_D$. On the one hand, since $i\leq j_*$, by the order given on the neighbors of $u$, we have $L(v_i)\geq L(v_{j_*})$. On the other hand, since the path $P_D$ is level-decreasing, $ L(v_{j_*})> L(v_i)$, leading a contradiction. So Algorithm~\ref{alg:recoloring} is correct.\vspace{10pt}

Let us finally prove Lemma~\ref{lem:notmanyrecol}. If we call Algorithm~\ref{alg:recoloring} with the initial coloring $\alpha$ and the list $P=\{v\}$ then it provides a sequence of proper colorings such that the color of $v$ in the final coloring is distinct from the initial one and no other vertex with level at least $L(v)$ has been recolored.  This concludes the proof of Lemma~\ref{lem:notmanyrecol}.
\end{proof}

The next lemma is a natural consequence of Lemma~\ref{lem:notmanyrecol}.
\begin{lemma}\label{lem:recol}
Suppose that a graph $G$ on $n$ vertices admits a $t$-partition of degree $\ell$. Then, for any $(\ell+2)$-coloring $\alpha$ there exists a $(\ell+1)$-coloring $\beta$ (that is, $\beta(v)\neq \ell+2$ for every $v\in V(G)$) such that $d(\alpha,\beta)\leq \ell^t n^2$.

Moreover, there exists a stable set $S$ such that $G\setminus S$ admits a $t$-partition of degree $\ell-1$.
\end{lemma}
\begin{proof}
Let us fix a $t$-partition of degree $\ell$ of $G$ and denote by $V_1,\ldots,V_t$ its parts.
By Lemma~\ref{lem:notmanyrecol}, we can change the color of every vertex in $v\in V_i$ with color $\ell+2$ by performing at most $\ell^i$ recolorings for each vertex in $\cup_{j < i} V_j$. Thus, first remove color $\ell+2$ from $V_t$ by recoloring each vertex in $G$ at most $\ell^t|V_t| $ times, then remove it from $V_{t-1}$ by recoloring each vertex at most $\ell^{t-1}|V_{t-1}| $, and so on. By Claim~\ref{lem:notmanyrecol}, while removing color $\ell+2$ from $V_i$, we do not recolor any of the vertices in $G_i$ (apart from the ones with color $\ell+2$). Therefore, while recoloring $V_i$ we never create new vertices in color $\ell+2$ in $G_i$. After removing color $\ell+2$ from $V_1$ we have a proper coloring $\beta$ of $G$ that does not use the color $\ell+2$. Moreover, we have recolored each vertex at most
$$
\ell  |V_1| + \ell^2 |V_2| + \dots + \ell^t |V_t|\leq \ell^t n\;,
$$
times. Thus the total number of recolorings is at most $\ell^t n^2 $ concluding the first part of the lemma.

It only remains to show that we can select a stable set such that the remaining graph has a $t$-partition of degree $\ell-1$. Let $S_t$ be a maximal (by inclusion) stable set in $G_t$. Define recursively $S_i$ to be a maximal (by inclusion) stable set in $G_i\setminus T_i$, where $T_i=\bigcup_{j\geq i} \left(S_j \cup N_{G_{i}}(S_j)\right)$~(recall that $N(X)$ is the set of vertices in $V\setminus X$ at distance one from some vertex in $X$) and let $S=S_1\cup \dots\cup S_t$. By construction of $T_i$, any vertex in $S_i$ is not in the neighborhood of $S_j$ for any $j>i$, thus $S$ is a stable set.

We claim that $\{V_1\setminus S_1,\dots, V_t\setminus S_t\}$ is a $t$-partition of degree $\ell-1$ of $G\setminus S$. We just need to show that every $v\in V_i\setminus S_i$ has degree at most $\ell-1$ in $G_i'= G_i\setminus S$. By the maximality condition of the selected stable sets, any such $v$ has at least one neighbor in $S$. In particular, by the order of the construction (from $V_t$ to $V_1$), it has at least one neighbor in $\cup_{j\geq i} S_j$ (otherwise $v$ could be included in $S_i$, contradicting the maximality of it).
Since $\{V_1,\dots, V_t\}$ is a $t$-partition of degree $\ell$, any $v\in V_i$ has at most $\ell$ neighbors in $G_i$. Therefore the degree of $v$ in $G_i'$ is at most $\ell-1$ and $G\setminus S$ admits a $t$-partition of degree $\ell-1$.
\end{proof}

Now we are ready to prove the main theorem of this section,
\begin{proof}[Proof of Theorem~\ref{thm:mad_recol}]
We will show that there exists a constant $c=c(d,\eps)$ such that any $k$-coloring $\alpha$ can be reduced to a canonical $k$-coloring $\gamma^*$ using $O(n^c)$ recoloring steps. This canonical coloring $\gamma^*$ only depends on structural properties of $G$ and not on the coloring $\alpha$ (the precise definition of $\gamma^*$ will be detailed below). The previous claim implies the statement of the theorem: between any pair of colorings $\alpha$ and $\alpha'$ there exists a path in the $k$-recoloring graph of length $O(n^c)$ (which in particular goes through $\gamma^*$).

Let us first use the fact that $\mad(G)\leq d-\eps$ to show that $G$ admits a $t$-partition of degree $d-1$ for some $t= O(\log{n})$. By the definition of the maximum average degree, every nonempty subgraph of $G$  has density at most $d-\eps$. Partition the set $V=U_{< d}\cup U_{\geq d}$ in two parts where $v\in U_{< d}$ if the degree of $v$ at most $d-1$ and $v\in U_{\geq d}$ otherwise.
We have,
$$
(d-\eps) \cdot n\geq  2|E(G)|= \sum_{v\in V} \deg(v) \geq \sum_{v\in U_{\geq d}} \deg(v)\geq d |U_{\geq d}|\;.
$$
This directly implies that $|U_{\geq d}|\leq \frac{d-\eps}{d}\cdot n$. Set the first part of the $t$-partition as $V_1=U_{< d}$. Notice that $|V_1| \geq \frac \epsilon d \cdot n$. Since the graph $G_2=V\setminus V_1$ is a subgraph of $G$, its maximum average degree is at most $d-\eps$ and thus we can repeat the same procedure on it. Moreover, $|V(G_2)|\leq \frac{d-\eps}{d}\cdot n$. After $m$ iterations of this procedure, we have $|V(G_m)|\leq \Big(\frac{d-\epsilon}{d}\Big)^m n$, and thus, we have to repeat this procedure at most $t=\log_{\frac{d}{d-\eps}}{n} = \frac{\log_{d}{n}}{\log_d (d/(d-\eps))}$ times before we finish the construction of the partition of degree $d-1$. Set
\[ c(d,\eps)= \frac{1}{\log_d (d/(d-\eps))}+2\;.\]
Recall that $G$ admits a $\left((c(d)-3)\cdot \log_d{n}\right)$-partition of degree at most $d-1$.

Let us show how to transform $\alpha$ into the canonical coloring $\gamma^*$. Let $G_{d-1}=G$ and $\alpha_{d-1}=\alpha$. For every $\ell$ from $d-1$ to $1$, we do the following recoloring procedure:
\begin{enumerate}
\item Use Lemma~\ref{lem:recol} to $G_\ell$ in order to transform the $(k-(d-1)+\ell)$-coloring $\alpha_{\ell}$ into a $(k-d+\ell)$-coloring $\beta$ using at most $\ell^t |V(G_\ell)|^2$ many recoloring steps.
\item Let $S_\ell$ be the stable set of $G_\ell$ provided by Lemma~\ref{lem:recol}. Observe that $S_\ell$ does not depend on the coloring $\alpha_\ell$. Construct the $(k-(d-1)+\ell)$-coloring $\beta'$ from $\beta$ by recoloring the vertices in $S_\ell$ with color~$(k-(d-1)+\ell)$.
\item Consider the graph $G_{\ell-1}=G_{\ell}\setminus S$ and let $\alpha_{\ell-1}$ be the $(k-d+\ell)$-coloring obtained by restricting $\beta'$ into $G_{\ell-1}$. Notice that, by Lemma~\ref{lem:recol}, $G_{\ell-1}$ admits a $\left((c(d)-2)\cdot\log_{d}{n}\right)$-partition of degree $\ell-1$.
\end{enumerate}

By Lemma~\ref{lem:recol}, at Step $1$ of every iteration we perform at most $\ell^t \cdot |V(G_\ell)|^2\leq d^t n^2\leq n^{c}$ many recolorings. At Step $2$ of each iteration we perform at most $|S_\ell|\leq n$ many recolorings. Recall that the number of iterations is $d-1$. Thus, the number of recolorings during the recoloring procedure is at most $d ( n^{c}+n)$.

Let $\alpha_0$ be the $k$-coloring obtained at the end of the procedure. Since the set $S_\ell$ obtained at Step $2$ only depends on the graph $G_\ell$ and the selected $t$-partition of degree $(d-1)$ of the graph $G$ but not on the coloring $\alpha_\ell$, the coloring $\alpha_0$ restricted to $G\setminus G_0$, does not depend on $\alpha$. Indeed, all the vertices of $S_\ell$ are colored with color $(k-(d-1)+\ell)$ for every $\ell$ between $1$ and $d-1$. Moreover $G_0=G\setminus (S_1\cup\dots\cup S_{d-1})$, has a $t$-partition $\{V_1, \dots, V_t\}$ of degree $0$, or, in other words, $G_0$ is the empty subgraph. Hence, $\alpha_0$ can be transformed into $\gamma^*$ by recoloring all the vertices in $G_0$ with color $1$ (in fact, only $d$ colors are used in $\gamma^*$). This can be done in at most $n$ recoloring steps.

Thus, we can transform any $k$-coloring $\alpha$ into a canonical $k$-coloring $\gamma^*$ (\emph{i.e.} a coloring that does not depend on $\alpha$) using at most $d(n^{c}+n)+n =O(n^c)$ many recolorings. This implies that for any two $k$-colorings $\alpha$ and $\alpha'$, we have $d(\alpha,\alpha')= O(n^c)$. Indeed, $\alpha$ can be transformed into $\gamma^*$ with at most $O(n^c)$ recolorings and $\alpha'$ can be transformed into $\gamma^*$ with at most $O(n^c)$ recolorings. Therefore,
$$
\diam(\C_{k}(G))= O(n^c)\;,
$$
concluding the proof of the theorem.
\end{proof}

We did not make any attempt to improve the constant $c$ obtained in Theorem~\ref{thm:mad_recol}. However, this constant can be decreased if we are more careful. For instance, the $n^2$ factor obtained in Lemma~\ref{lem:recol} can be replaced by $n$, since Claim~\ref{cla:numberpaths} actually bounds the number of decreasing paths between $w$ and vertices at the same level as $u$ (if we assume that $L(w)<L(u)$).

Note that the proof also provides an algorithm which runs in polynomial time. Indeed Algorithm~\ref{alg:recoloring} runs in polynomial time. Moreover the partition of Theorem~\ref{thm:mad_recol} can be found in polynomial time as well as the stable set provided by Lemma~\ref{lem:recol}. So the proof gives an algorithm which transforms any $k$-coloring into any other one in polynomial time, provided that $\mad(G)\leq d-\eps$ for some $\eps>0$, and that $k\geq d+1$.

\section{Recoloring planar graphs and related classes}

As observed in~\cite{BonamyB14}, there is a planar graph $G$ (the graph of the icosahedron, see Figure~\ref{fig:6}) such that $\C_6(G)$ is not even connected ($\diam(\C_{6}(G))=\infty$)\;. There also exists a planar graph $G$ such that  $\C_5(G)$ is not connected ($\diam(\C_{5}(G))=\infty$)\; (for instance consider the graph of Figure~\ref{fig:6} where vertices colored with $6$ were deleted). In both cases the reason is the same: the colorings are frozen and then no vertex can be recolored, or, otherwise stated, the coloring is an isolated vertex in the recoloring graph.

\begin{figure}[ht!]
 \begin{center}
 \includegraphics[width=0.4\textwidth]{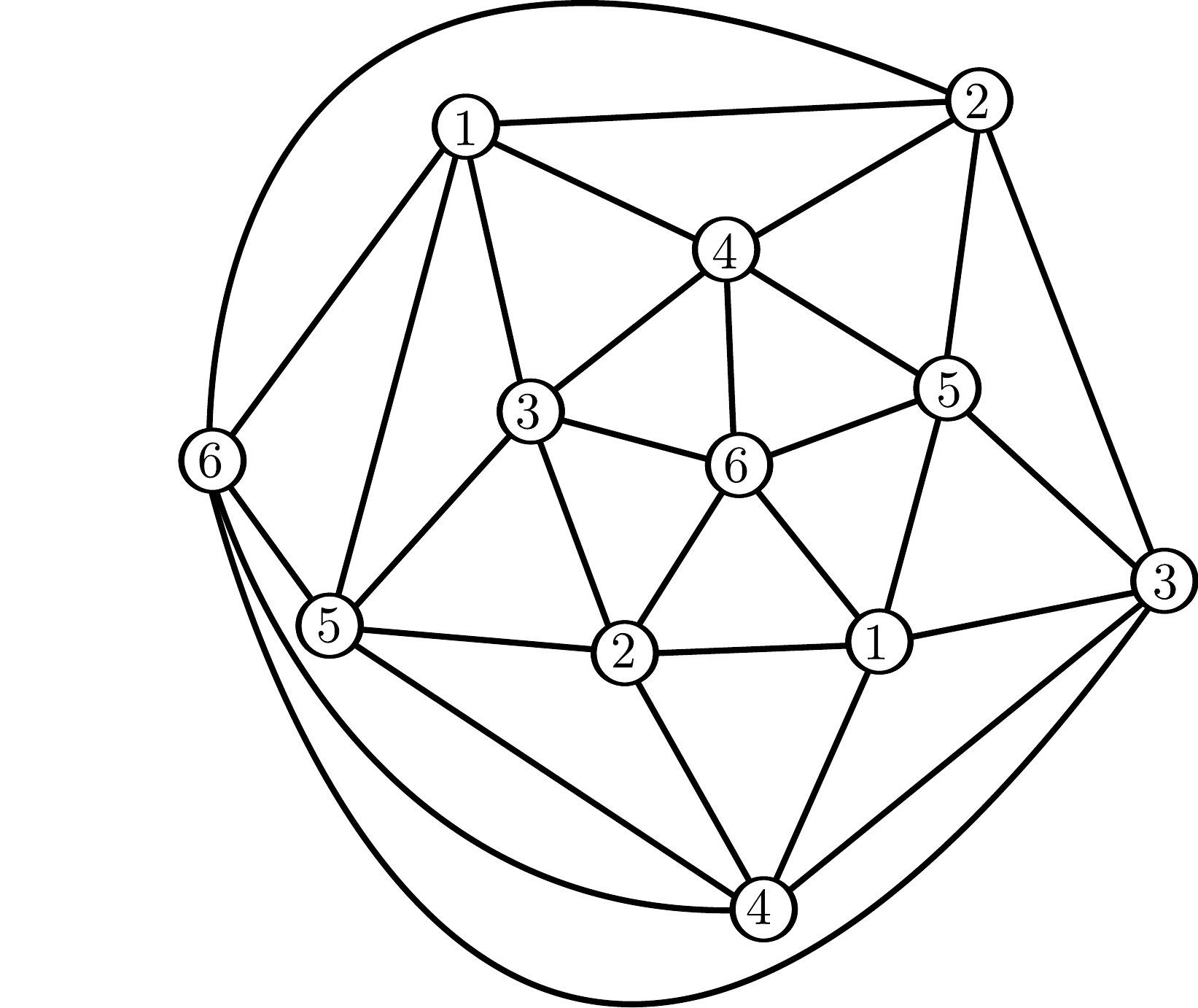}
 \end{center}
 \caption{A $6$-coloring corresponding to an isolated vertex in $\C_6(G)$.}
 \label{fig:6}
\end{figure}

Recall that any planar graph $G$ is $5$-degenerate. The result of Cereceda~\cite{Cereceda} on the degeneracy of implies that for any planar graph $G$, $\diam(\C_{11}(G))=O(n^2)$\;. The result of Dyer et al~\cite{dyer2006randomly} show that $\C_{k}(G)$ is connected for every $k\geq 7$\;. The best known upper bound for the diameter in the cases $k=7,8,9,10$ is the trivial one due to Dyer et al.~\cite{dyer2006randomly}, \emph{i.e.} $\diam(\C_{k}(G))\leq k^n$\;.

As a corollary of Theorem~\ref{thm:mad_recol}, we obtain that $\C_{8}(G)$ has polynomial diameter.
\begin{corollary}
For any planar graph $G$ on $n$ vertices and any $k\geq 8$,
$$
\diam(\C_{k}(G))=\Poly(n)\;.
$$
\end{corollary}
\begin{proof}
Euler formula ensures that for every planar graph $H$, $ |E(H)| \leq 3 |V(H)| - 6$.
Since every subgraph of a planar graph is also planar, we have $\mad(G)<6$. So we just have to apply Theorem~\ref{thm:mad_recol} with $d=7$ and $\eps=1$ to conclude.
\end{proof}

It would be interesting to determine whether $\diam(\C_{7}(G))=\Poly(n)$ or not. Observe that while Theorem~\ref{thm:mad_recol} is be able to prove such statement for a graph $G$ with $\mad(G)=5.99$, it is not enough to prove it for a planar graphs because their maximum average degree is not bounded away from $6$. Unfortunately, the same partition argument we used for the proof of Theorem~\ref{thm:mad_recol} will not be able to show that the diameter is small in the case we use $7$ colors. Here we briefly sketch the argument
\begin{proposition}\label{prop:degree6}
There exists a planar graph $G$ on $n$ vertices that does not admit any $\frac{\sqrt{n}}{2}$-partition of degree~$5$.
\end{proposition}
\begin{proof}
\begin{figure}
 \centering
 \includegraphics[scale=0.6]{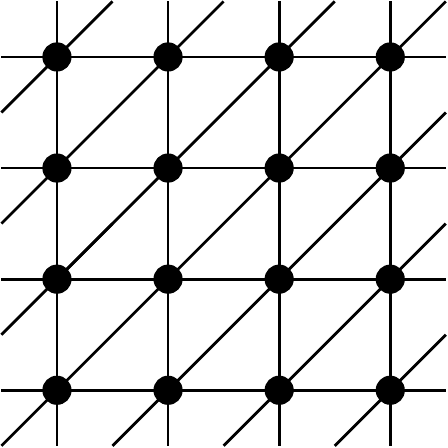}
 \caption{The structure of a planar graph with no $\frac{\sqrt{n}}{2}$-partition of degree~$5$.}
 \label{fig:supergrid}
\end{figure}

Suppose that $n=4m^2$ and let $G$ be the graph with vertex set $V(G)=\{(i,j):1\leq i,j \leq 2m\}$ and edge set $E(G)=\{(i_1,j_1)(i_2,j_2): |i_1-i_2|+|j_1-j_2|=1\}\cup \{(i_1,j_1)(i_2,j_2): i_1=i_2+1,\;j_1=j_2+1\}$. This can be seen as a triangulated grid with \emph{inner vertices} (\emph{i.e.} vertices with both coordinates in $\{2,\dots,2m-1\}$) of degree $6$~(see Figure~\ref{fig:supergrid}).

We claim that $v=(m,m)\notin V_i$, for any partition of degree $5$ and $i<m$. We show it by induction in $m$. For $m=1$ there is nothing to prove. Since any inner vertex has degree $6$, for any such partition, $V_1$ does not contain inner vertices. We can assume that $V_1$ is composed by all the vertices of degree at most $5$ in $G$, that is the ones lying on the boundary of the grid. Now, $G_2=G\setminus V_1$ is a $2(m-1)\times 2(m-1)$ triangulated grid. Thus, by induction hypothesis, the vertex $m$ is not in $V_i$'s of $G_2$ for all the $i<m-1$. This proves the claim.
\end{proof}
Closing the gap between $7$ and $8$ on planar graphs is an interesting open problem which may give new methods for tackling Cereceda et al.'s degeneracy conjecture.
Moreover note that since the graph presented in Proposition~\ref{prop:degree6} is $3$-colorable, the method introduced for Theorem~\ref{thm:mad_recol} is not useful to prove that the diameter of $\C_7(G)$ is polynomial even if $G$ is a $3$-colorable planar graphs.

Though, an interesting result can be obtained for triangle-free planar graphs (recall that triangle-free planar graphs are $3$-colorable by Gr\"otzsch's theorem).
\begin{corollary}\label{prop:degree4}
For any triangle-free planar graph $G$ on $n$ vertices and any $k \geq 6$ we have
$$
\diam(\C_{k}(G))=\Poly(n)\;.
$$
Besides, there exists a triangle-free planar graph $G$ on $n$ vertices that does not admit any $\frac{\sqrt{n}}{2}$-partition of degree~$4$.
\end{corollary}
\begin{proof}
Again, a slight variant of the Euler formula ensures that for every triangle-free planar graph $H$, $ |E(H)| \leq 2|V(H)| - 4$. Since every subgraph of a triangle-free planar graph is also triangle-free and planar, we have $\mad(G)<4$. So we just have to apply Theorem~\ref{thm:mad_recol} with $d=5$ and $\eps=1$ to conclude.


For the second part of the statement, suppose that $n=4m^2$ and let $G$ be the graph with vertex set $V(G)=\{(i,j):1\leq i,j \leq 2m\}$ and edge set $E(G)=\{(i_1,j_1)(i_2,j_2): |i_1-i_2|+|j_1-j_2|=1\}$. This can be seen as a grid. We claim that $v=(m,m)\notin V_i$, for any partition of degree $3$ and $i<m$, which can be proved as in Proposition~\ref{prop:degree6}. So the argument cannot be extended to $5$-colorings of triangle-free planar graphs.
\end{proof}

\paragraph{Acknowledgement}
The authors want to thank Marthe Bonamy for fruitful discussions and for pointed out a weaker version of Theorem~\ref{thm:lineardiam}.

\bibliography{recoloring}

\begin{thebibliography}{10}

\bibitem{BonamyB13}
M.~Bonamy and N.~Bousquet.
\newblock {Recoloring bounded treewidth graphs}.
\newblock {\em Electronic Notes in Discrete Mathematics}, 44:257--262, 2013.

\bibitem{BonamyB14}
M.~Bonamy and N.~Bousquet.
\newblock {Recoloring graphs via tree decompositions}.
\newblock {\em CoRR}, 1403.6386, 2014.

\bibitem{BonamyB14a}
M.~Bonamy and N.~Bousquet.
\newblock {Reconfiguring independent sets in cographs}.
\newblock {\em CoRR}, 1406.1433, 2014.

\bibitem{BonamyJ12}
M.~Bonamy, M.~Johnson, I.~Lignos, V.~Patel, and D.~Paulusma.
\newblock {Reconfiguration graphs for vertex colourings of chordal and chordal
  bipartite graphs}.
\newblock {\em Journal of Combinatorial Optimization}, pages 1--12, 2012.

\bibitem{BonsmaC07}
P.~Bonsma and L.~Cereceda.
\newblock {Finding Paths Between Graph Colourings: {PSPACE}-Completeness and
  Superpolynomial Distances.}
\newblock In {\em {MFCS}}, volume 4708 of {\em {Lecture Notes in Computer
  Science}}, pages 738--749, 2007.

\bibitem{BonsmaMNR14}
P.~Bonsma, A.~Mouawad, N.~Nishimura, and V.~Raman.
\newblock {The Complexity of Bounded Length Graph Recoloring and CSP
  Reconfiguration}.
\newblock In {\em {to appear in IPEC'14}}, 2014.

\bibitem{Cereceda}
L.~Cereceda.
\newblock {\em {Mixing Graph Colourings}}.
\newblock PhD thesis, London School of Economics and Political Science, 2007.

\bibitem{Cereceda09}
L.~Cereceda, J.~van~den Heuvel, and M.~Johnson.
\newblock {Mixing 3-colourings in bipartite graphs}.
\newblock {\em Eur. J. Comb.}, 30(7):1593--1606, 2009.

\bibitem{CerecedaHJ11}
L.~Cereceda, J.~van~den Heuvel, and M.~Johnson.
\newblock {Finding paths between 3-colorings}.
\newblock {\em Journal of Graph Theory}, 67(1):69--82, 2011.

\bibitem{Diestel}
R.~Diestel.
\newblock {\em {Graph Theory}}, volume 173 of {\em {Graduate Texts in
  Mathematics}}.
\newblock Springer-Verlag, Heidelberg, third edition, 2005.

\bibitem{dyer2006randomly}
M.~Dyer, A.~D. Flaxman, A.~M Frieze, and E.~Vigoda.
\newblock {Randomly coloring sparse random graphs with fewer colors than the
  maximum degree}.
\newblock {\em Random Structures \& Algorithms}, 29(4):450--465, 2006.

\bibitem{Gopalan09}
P.~Gopalan, P.~Kolaitis, E.~Maneva, and C.~Papadimitriou.
\newblock {The Connectivity of Boolean Satisfiability: Computational and
  Structural Dichotomies}.
\newblock {\em SIAM J. Comput.}, pages 2330--2355, 2009.

\bibitem{ItoKO14}
T.~Ito., M.~Kaminski, and H.~Ono.
\newblock {Independent set reconfiguration in graphs without large bicliques }.
\newblock In {\em {ISAAC'14}}, 2014.

\bibitem{NesetrilPOM}
J.~Nesetril and P.~{Ossana de Mendez}.
\newblock {\em {Sparsity - Graphs, Structures, and Algorithms.}}, volume~28 of
  {\em {Algorithms and combinatorics}}.
\newblock Springer, 2012.

\bibitem{SuzukiMN14}
A.~Suzuki, A.~Mouawad, and N.~Nishimura.
\newblock {Reconfiguration of Dominating Sets}.
\newblock {\em CoRR}, 1401.5714, 2014.

\end{thebibliography}

\bibliographystyle{plain}

\end{document}